\documentclass[12pt,a4paper]{article}

\usepackage[T1]{fontenc}
\usepackage[utf8]{inputenc}
\usepackage[english]{babel}
\usepackage{microtype} 
\usepackage{amsmath}  
\usepackage{amsfonts} 
\usepackage{mathtools} 
\usepackage[colorlinks=true,allcolors=blue]{hyperref}
\usepackage{autonum}
\usepackage{bm} 
\usepackage{graphicx} 
\usepackage{eurosym}
\usepackage{amsthm}
\usepackage{verbatim} 
\usepackage{enumitem} 
\usepackage{color}
\usepackage[left=3cm, right=3cm, top=3cm, bottom=4cm]{geometry}

\usepackage{pgfplots}
\pgfplotsset{width=12cm,compat=1.9}

\theoremstyle{plain}
\newtheorem{thm}{Theorem}[section]

\newtheorem{lem}[thm]{Lemma}
\newtheorem{prop}[thm]{Proposition}

\theoremstyle{definition}
\newtheorem{defn}{Definition}[section]

\theoremstyle{remark}
\newtheorem{oss}{Remark}

\newcommand{\norm}[1]{\left\lVert#1\right\rVert}
\newcommand{\p}{\mathbb{P}}

\newcommand{\E}{\mathbb{E}}

\newcommand{\fip}[1]{\Phi_+(#1 -a_{_+})}
\newcommand{\fim}[1]{\Phi_-(#1 -a_{_-})}

\newcommand{\dpsip}[1]{\Psi'_+(#1)}
\newcommand{\dpsim}[1]{\Psi'_-(#1)}

\newcommand{\zeroinf}{(0,\infty)}
\newcommand{\mb}[1]{\mathbb{#1}}

\newcommand{\mc}[1]{\mathcal{#1}}
\newcommand{\m}{\overline{m}}
\newcommand{\tp}{\tilde{\p}}
\newcommand{\te}{\tilde{\E}}
\newcommand{\bpl}{\beta_{_+}}
\newcommand{\bmin}{\beta_{_-}}

\title{On a family of critical growth-fragmentation semigroups and refracted L\'evy processes}
\author{Benedetta Cavalli\footnote{Institut für Mathematik, Universität Zürich, Switzerland}}

\begin{document}
	\maketitle
	\vspace{-0.6 cm}

\begin{abstract}
The growth-fragmentation equation models systems of particles that grow and split as time proceeds. An important question concerns the large time asymptotic of its solutions.
Doumic and Escobedo ($2016$) observed that when growth is a  linear function of the mass and fragmentations are  homogeneous, the so-called Malthusian behaviour fails. 
In this work we further analyse the critical case by considering a piecewise linear growth, namely
\begin{equation}
c(x) =
\begin{cases}
a_{_-} x \quad \quad x < 1 \\
a_{_+} x \quad \quad x \geq 1,
\end{cases}
\end{equation}
with $0 < a_{_+} < a_{_-}$. We give necessary and sufficient conditions on the coefficients ensuring the Malthusian behaviour with exponential speed of convergence to an asymptotic profile,  and also provide an explicit expression of the latter. Our approach relies crucially on properties of so-called  refracted Lévy processes that arise naturally in this setting.
\end{abstract}

\vspace{0.4 cm}


\textit{Classification MSC:} 34K08 , 35Q92, 47D06, 47G20, 45K05, 60G51, 	60J99  

\normalsize

\section{Introduction} 
	To describe the growth and division of particles over time, one of the key equations in the field of structured population dynamics is the so-called \textit{growth-fragmentation equation}. This equation was first introduced at the end of the sixties to model cells dividing by fission \cite{BA67}, but it is also used to describe protein polymerization \cite{CLODLMP09}, neuron networks \cite{KPS13}, \cite{PPS13}, the TCP/IP window size protocol for the internet \cite{BMR02} and many other models. For all these applications, the common point is that the ‘particles’ under concern (which can be cells, polymers, dusts, windows, etc.) are well-characterized by their ‘size’, i.e. a one-dimensional quantity which grows over time at a certain rate depending on the size, and which is distributed among the offspring when the particle divides, in a way in which the total mass is conserved. In this work, we assume that the dislocations are \textit{homogeneous}. This means, roughly speaking, that particles dislocate at a constant rate $K>0$ independently of their sizes and that the distribution of the ratios of the size daughter/mother does not depend on the size of the mother.  \\
The population is thus described by the concentration of particles of size\footnote{In the sequel we will refer to it also using the term ‘mass’.} $x > 0$ at time $t \geq 0$, denoted by $u_t(x)$. 
The evolution of $u_t(x)$ is governed by the following integro-partial differential equation, which can be obtained either by a mass balance, in a similar way as for fluid dynamics (\cite{BSC11}, \cite{MD86}), or by considering the Kolmogorov equation for the underlying jump process  (\cite{CLOEZ17}, \cite{DHNR15}):
\begin{equation}
\label{gfehomog}
\partial_t u_t(x) + \partial_x(c(x)u_t(x))= \int_0^{1} u_t \Big(\frac{x}{s} \Big) s^{-1} \rho (s) ds -  K u_t(x),
\end{equation}
where the initial condition $u_0$ is prescribed. \\
The function $c:(0,\infty) \to (0,\infty)$ is the so-called \textit{growth rate} and it expresses the rate at which each particle grows according to its size. In the literature, $c$ has been widely considered to be a continuous function and some of the most extensively studied cases are those for which $c$ is constant or $c$ is a linear function. \\
The function $\rho : [0,1] \to [0, \infty)$, represents the so-called \textit{fragmentation kernel} and gives the rate at which a particle of mass $sx$ appears as a result of the dislocation of a particle of mass $x$. We assume that
\begin{equation}
\label{conditiononrhonecessary}
\int_0^1 s^{- \epsilon}  \rho(s) ds < + \infty,
\end{equation} 
for some $\epsilon > 0$. The conservation of mass at dislocation events gives the identity
\begin{equation}
K  = \int_0^1 s\rho(s) ds. 
\end{equation}

Due to the wide range of applications in mathematical modelling, existence, uniqueness and long term behaviour of the solutions of the growth fragmentation equation have been studied over many years. One of the major issues consists in the study of the asymptotic behaviour of the solutions $u_t$.} As far  as the asymptotic behaviour is concerned, it was proved that, under fairly general balance assumptions on the growth and the fragmentation rates, the population grows exponentially over time but tends to a steady profile. In other words, under such assumptions, there exists a constant $\lambda$, called the \textit{Malthus exponent}, for which $e^{- \lambda t}u_t$ converges, in some suitable space, to a so-called \textit{asymptotic profile} $\nu$. Remarkably, this trend to equilibrium is not only a mathematical result but is also supported by biological evidence [27], which may also explain the very fast de-synchronization of cells, for instance [7]. Ideally, one also wishes to have informations about the \textit{rate of convergence}, i.e. ensuring the existence of some $\beta > 0$ such that $e^{\beta t} (e^{- \lambda t}u_t - \nu) $ converges to $0$.\\
A key instrument is the so-called \textit{growth-fragmentation operator}, which, in our case, has the form
\begin{equation}
\label{gfoperator}
\mathcal{A} f (x)  = 
c(x) f'(x) + \int_0^1 (f(sx) - sf(x))  \rho(s) ds, \quad \quad x >0,
\end{equation}
and is defined for smooth compactly supported $f$, say. \\ The weak form of the growth-fragmentation equation $\eqref{gfehomog}$ is
\begin{equation}
\label{wgfe}
\frac{d}{dt} \langle u_t, f \rangle = \langle u_t, \mc A f \rangle, 
\end{equation}
where $\langle \mu , g \rangle$ denotes $\int g(x) \mu(dx)$ for any measure $\mu$ and any function $g$, whenever the integral makes sense. If $h$ is a non negative measurable function, we define $\langle h , g \rangle \coloneqq \langle \mu , g \rangle$ with $\mu(dx) = h(x) dx$.
Under quite simple general assumptions on the rates $c$ and $\rho$, $\mc A$ represents the infinitesimal generator of a unique strongly continuous positive semigroup $(T_t)_{t \geq 0}$ and the solution of \eqref{wgfe} can be represented as
\begin{equation}
\label{weaksolutiongfe}
\langle u_t, f \rangle = \langle u_0, T_tf \rangle.
\end{equation}
In general, there is no explicit expression for the growth-fragmentation semigroup $(T_t)_{t \geq 0}$ and many works are concerned with its asymptotic behaviour. 
Typically, one expects, under proper assumptions on the growth and fragmentation rates, that there exists a \textit{leading eigenvalue} $\alpha \in \mathbb{R}$ such that
\begin{equation}
\label{convergenceofthesemigroup}
\lim_{t \to \infty} e^{-\alpha t} T_t f(x) = h(x) \langle \nu, f \rangle, \quad x > 0,
\end{equation}
at least for every continuous and compactly supported function $f: (0, \infty) \to \mathbb{R}$. Here, $\nu(dx)$ is a Radon measure, commonly called \textit{asymptotic profile}, and $h$ a positive function. 

In the literature, the above convergence is often referred to as \textit{Malthusian behaviour}. When it  holds, it is furthermore important to estimate the speed of convergence. In fact, say for $\alpha > 0$, an indefinite exponential growth is unrealistic in practice due to several effects such as competition between individuals for space and resources. As a consequence, the growth fragmentation equation can be reliable only for describing rather early stages of the evolution of the population, and the notions of leading eigenvalue and asymptotic profile are meaningful only when convergence  \eqref{convergenceofthesemigroup} occurs fast enough. 

One of the main tools that have been used for establishing the validity of \eqref{convergenceofthesemigroup} is the spectral theory of semigroups and operators. Several authors have shown, under proper assumptions on the growth and fragmentation rates, the existence of positive eigenelements\footnote{The existence of eigenelements has been proved also in the case of much more general generators. In particular, the fragmentation kernel does not need to be homogeneous.} associated to the leading eigenvalue of the operator $\mc A$ and its dual $\mc A^*$, namely a Radon measure $\nu$ and a positive function $h$ such that for some $\alpha \in \mathbb{R}$
\begin{equation}
\label{eigenelements}
\mc A h = \alpha h, \quad \mc A^* \nu = \alpha \nu, \quad \text{and} \quad  \langle \nu, h \rangle = 1.
\end{equation} 

Notably, specific assumptions on the growth and fragmentation rates that ensure existence and uniqueness of a positive leading eigenvalue and positive eigenfunctions have been obtained by Mischler \cite{M06}, Doumic and Gabriel \cite{DG10} and Mischler and Scher in \cite{MS16}. \\ Once the existence of positive eigenelements has been proved, several techniques can be used to prove the convergence. Cáceres at al. \cite{CCM11} used dissipation of entropy and entropy inequalities methods to establish the exponential convergence to an asymptotic profile for constant and linear growth rate. Perthame \cite{PERTHAME07} and Mischler et al. \cite{MMP05} established exponential convergence using the general relative entropy method. Mischler and Scher in \cite{MS16} provided a punctual survey on the spectral analysis of semigroups and they developed a splitting technique that allows them to formulate a Krein-Ruttman theorem and to establish exponential rate of convergence. The exponential rate of convergence is strongly related to the existence of a spectral gap. For example, results of this type can be found in Perthame and Ryzhik \cite{PR05}, Laurençot and Perthame \cite{LP09}, Cáceres at al. \cite{CCM11} and Mischler and Scher \cite{MS16}.

A more direct approach relying on the Mellin transform has been used by Doumic and Escobedo \cite{DE16} and Bertoin and Watson \cite{BW16} to analyse the so-called \textit{critical case}, in which the strategy outlined above cannot be applied as it is not possible to find a solution for the eigenvalue problem \eqref{eigenelements}. Indeed, in this case, even though one can find eigenelements for the growth-fragmentation operator, the integrability condition $\langle \nu, h \rangle = 1$ is not satisfied and the convergence \eqref{convergenceofthesemigroup} fails. More specifically, Doumic and Escobedo \cite{DE16} studied in depth the case in which the growth-rate is linear and the fragmentation kernel is homogeneous and Bertoin and Watson \cite{BW16} extended the analysis to the general self-similar case, which includes a broader range of growth and fragmentation rates. The linear case is also analysed in Section $6$ of \cite{BW18}. 

In this work, we go further in the analysis of the critical case. We consider a piecewise-linear growth-rate, namely
\begin{equation}
\label{growthrate}
c(x) =
\begin{cases}
a_{_-} x \quad \quad x < 1 \\
a_{_+} x \quad \quad x \geq 1,
\end{cases}
\end{equation}
with $0 < a_{_+} < a_{_-}$ and we investigate whether \eqref{convergenceofthesemigroup} may occur provided that $a_{_+}$ and $a_{_-}$ are suitably chosen. 
The motivations behind this choice are not only related to its interest from a mathematical point of view, but also to its interest in the applications. In words, small particles (i.e. with size smaller than $1$) grow at a faster rate than the larger ones. 

We follow the probabilistic approach of Bertoin and Watson in \cite{BW18}, which circumvents the spectral theory of semigroups.
Their approach relies on a Feynman-Kac representation of the semigroup $(T_t)_{t \geq 0}$ in terms of an instrumental Markov process $X=(X_t)_{t \geq 0}$, whose infinitesimal generator $\mc G$ is closely related to the growth-fragmentation operator $\mc A$.
to ensure the convergence of $e^{-\lambda t} T_t$ and to give an expression for the asymptotic profile $\nu$ . Moreover, Bertoin in \cite{BERT18} further shows that such condition is also necessary for the Malthusian behaviour. \\

When $c$ is linear (see Section $6$ of \cite{BW18}), $X$ is the exponential of a Lévy process. In the present setting, the key point that enables us to make the general approach more detailed is that the process $(X_t)_{t \geq 0}$ is the exponential of a so-called \textit{refracted Lévy process} $(\xi_t)_{t \geq 0}$. Heuristically, the refracted Lévy process that we take into consideration has the same jumps as a usual Lévy process but its drift changes depending on whether the process is positive or negative. We refer to Section $2$ for a rigorous definition of refracted Lévy processes. A more general introduction about the topic as well as a comprehensive analysis of their occupation times was given by Kyprianou et al. in \cite{KPP14}. Important references are also \cite{KL10} and \cite{RENAUD14}. \\
The first contribution of the present work is to provide explicit criteria in terms of the coefficients $\rho$ and $a_{\pm}$ to ensure exponentially fast convergence towards the asymptotic profile. 
\begin{thm}
\label{Teorema1}
Assume that \eqref{conditiononrhonecessary} and \eqref{growthrate} hold and let 
\begin{equation}
\lambda = \int_0^1 (1-s) \rho(s) ds.
\end{equation}
\begin{enumerate}
\item 
If 
\begin{equation}
\label{existencehyp}
a_{_+} < - \int_0^1 \log (s) \;  \rho(s) ds < a_{_-}
\end{equation}
then, there exist a probability measure $\nu(dx)$ and $\epsilon > 0$ such that for every continuous function $f$ with compact support and for every $x > 0$,
\begin{equation}
e^{- \lambda t} T_t f (x) = \langle \nu, f \rangle + o( e^{- \epsilon t}) \quad \quad \text{as} \;  t \to \infty.
\end{equation}
\item If one of the two inequalities in \eqref{existencehyp} does not hold, then the Malthusian behaviour \eqref{convergenceofthesemigroup} fails.
\end{enumerate}
\end{thm}
The second contribution of this work concerns the asymptotic profile $\nu$ of the solutions. Usually, even when it is possible to establish  \eqref{convergenceofthesemigroup}, it is difficult to find an explicit expression for the asymptotic profile.  In \cite{BCG13}, the authors provided fine estimates on the principal eigenfunctions, giving their first order behaviour close to $0$, and $+ \infty$ (see also \cite{BW16}).\\
In this work, we have been able to characterize the asymptotic profile in a fairly explicit way using special properties of refracted Lévy processes. 

The main result is that the asymptotic profile $\nu(dx)$ is absolutely continuous with respect to the Lebesgue measure and it is possible to characterize the behaviour of its density, which we denote $\nu(x)$, for $x > 0$. \\
In particular, assuming that \eqref{existencehyp} holds, 
\begin{equation}
\nu(x) = \frac{c_1}{a_{_+}} x^{\tiny{- \left( 1 + \bpl \right)}}, \quad \quad x \geq 1,
\end{equation}
where $ \beta_{_+}$ and $c_1$ are positive parameters that will be characterized later (respectively in \eqref{bplus} and \eqref{normalizingc}). \\
For $x < 1$, it is not always possible to obtain an explicit expression for the density $\nu(x)$. However, assuming that \eqref{existencehyp} and the so-called \textit{Cramér's condition} \eqref{Cramerscondition} hold, we show that
\begin{equation}
\nu(x) \sim c_2 x^{ - 1 + \bmin}, \quad \quad x \rightarrow 0,
\end{equation}
where $\bmin$ and $c_2$ are positive parameters that will be characterized later (respectively in \eqref{Cramerscondition} and \eqref{normalizingc2}).
In some cases, much more can be said about the density. As an example, this happens when the fragmentation kernel is given by $\rho(x) = x^{\gamma - 1}$, for some $\gamma \geq 1$. In this case, the density can be computed also for $x < 1$ and it holds
\begin{equation}
\nu(dx)= c_3 \left( \frac{1}{a_{_-}}  x^{-(1 - \bmin)} \mathbf{1}_{\{0 < x <1\}} +  \frac{1}{a_{_+}}  x^{-(1 + \bpl)} \mathbf{1}_{\{x \geq 1\}} \right) dx,
\end{equation}
where $c_3$ is a positive constant that will be characterized in Section $5$.

The rest of the article is organised as follows. 

In Section \ref{Section2} we analyse the process $X$ and its relation with the growth-fragmentation operator $\mc A$ and the growth-fragmentation semigroup $T$. We introduce the class of refracted Lévy processes and we give an explicit representation of $X$ as the exponential of a refracted Lévy process $\xi$. 

Section \ref{Section3} is devoted to the proof of Theorem \ref{Teorema1}. 

In Section \ref{Section4} we give an explicit expression for the asymptotic profile $\nu$. 

In Section \ref{section5} we present some examples, concerning particular choices of the fragmentation kernel.

\section{Preliminaries and general strategy}
\label{Section2}
Let $L^{\infty}$ be the Banach space of measurable and bounded functions $f: (0, \infty) \to \mathbb{R}$, endowed with the supremum norm $\norm{\cdot}_{\infty}$. It is also convenient to set $\underline{f}(x)= x^{-1} f(x)$ for every $f \in L^{\infty}$ and $x > 0$ and define $\underline{L}^{\infty} = \{ \underline{f} \; : \; f \in L^{\infty} \}$. 

The first result concerns the existence and uniqueness of a semigroup $(T_t)_{t \geq 0}$ whose infinitesimal generator coincides with $\mc A$. 

For the proof, we refer to the one of Lemma $2.1$ in Section $2$ of \cite{BW18}, which can be adapted to our framework even though $c$ is not continuous. More precisely, since
\begin{equation}
\label{sublinearityofc}
\norm{\underline{c}}_{\infty} \coloneqq \sup_{x > 0} c(x)/x < \infty.
\end{equation}
it is possible to apply \cite[Theorem 8.3.3]{EK86} and obtain the following lemma. 

\begin{lem}
Under the assumptions \eqref{conditiononrhonecessary} and \eqref{growthrate}, there exists a unique positive strongly continuous semigroup $(T_t)_{t \geq 0}$ on $L^{\infty}$ with infinitesimal generator $\mc A$. 
\end{lem}

Following the probabilistic approach developed by Bertoin and Watson in \cite{BW18}, we provide a probabilistic representation of the main quantities of interest, such as the semigroup $T_t$, the Malthus exponent $\lambda$ and the asymptotic profile $\nu$, in terms of  $(X_t)_{t \geq 0}$. 
The analysis of the case of linear growth-rate (\textit{i.e.} $a_{_+}= a_{_-}=a$) and homogeneous fragmentation kernel made by Bertoin and Watson in Section $6$ of \cite{BW18} relies on the fact that the process $X$ is the exponential of a Lévy process. In the present setting, due to the piecewise linear growth-rate $c$, the process $X$ will be shown to be the exponential of a refracted Lévy process. 

Refracted Lévy processes arise by a simple perturbation of the paths of a Lévy process, which consists in subtracting off a fixed linear drift whenever the aggregate process is above a certain level. Whenever it exists, a refracted Lévy process $Z$ is described by the unique solution to the stochastic differential equation 
\begin{equation}\label{sderefracted}
dZ_t = - \delta \mathbf{1}_{ \{ Z_t > r \}} dt + d \tilde{Z}_t,
\end{equation}
where $\tilde{Z}= (\tilde{Z}_t, t \geq 0)$ is a Lévy process with only negative jumps and $\delta > 0$ is such that the resulting process $Z$ visits the half line $(r, +\infty)$ with positive probability. The generator of a refracted Lévy process can be expressed as
\begin{equation}
\label{generatorefractedlevy}
\mc G_Z f(x) =  \mc G_{\tilde{Z}} f(x) - \delta f'(x) \mathbf{1}_{[r, \infty)}(x).
\end{equation}

A comprehensive introduction to refracted Lévy processes as well as a complete analysis of their existence can be found in \cite{KL10}, and an analysis of their occupation times was made by Kyprianou et al. in \cite{KPP14}.   

We introduce the operator 
\begin{equation}
\mathcal{G} f(x)  = 
\begin{cases}
a_{_-} x f'(x) + \int_0^1 (f(sx)-f(x)) s \rho(s) ds   \quad 	\quad x <1  \\
a_{_+} x f'(x) + \int_0^1 (f(sx)-f(x)) s \rho(s) ds  \quad   \quad   x \geq 1.
\end{cases}
\end{equation}
Comparing $\mc G$ with \eqref{generatorefractedlevy} and with the expression of the generator of a Lévy process (see for example p. $24$ of \cite{BLP}), it is straightforward that there exists a refracted Lévy process $(\xi_t)_{t \geq 0}$ such that $\mathcal{G}$ is the generator of the Markov process 
\begin{equation}
\left(X_t\right)_{t \geq 0} \coloneqq \left(e^{\xi_t}\right)_{t \geq 0}.
\end{equation}
We denote by $\p_x$ and $\E_x$ the law and the corresponding expectation of $X$ starting from $x > 0$. 
To describe $\xi$, we consider a Lévy process $\xi^-$ composed of a compound Poisson process with only negative jumps plus a linear drift with rate $a_{_-}>0$, with Lévy measure given by
\begin{equation}
 \Pi(dz)= e^{2z} \rho (e^z) \; dz, \quad \quad z<0.
 \end{equation}
Since  $\xi^-$ is a spectrally negative Lévy process, we can define its Laplace exponent $\Psi_{_-}  : [0, \infty) \to \mathbb{R}$ by
\begin{equation}
\label{spectrallynegativelevyprocesses}
\E\left[\exp \{q \xi^-_t \} \right] = \exp \{t \Psi_{_-}(q)\}, \quad \quad t, q \geq 0.
\end{equation}
The Lévy-Khintchine formula gives
\begin{equation}
\label{psim}
 \Psi_{_-}(q) = a_{_-}q + \int_0^1 (s^q -1) s \rho (s) ds. 
\end{equation}
Then, $\xi$ is the solution to the stochastic differential equation
\begin{equation}
\label{xixixi}
d\xi_t = - (a_{_-} - a_{_+}) \mathbf{1}_{ \{ \xi_t > 0 \}} dt + d \xi^-_t.
\end{equation}
\begin{figure}
\centering
\begin{tikzpicture}
\begin{axis}[
    axis lines = left,
    xlabel = $t$,
    ylabel = {$\xi(t)$},
    axis x line=center,
    legend pos=south east
]
\addplot [
    domain=0:1.2, 
    samples=100, 
    color=purple,
]
{2.5*x};
\addlegendentry{$\xi^+$}
\addplot [
    domain=2.1:3.6, 
    samples=100, 
    color=blue,
    ]
    {5*x - 20};
\addlegendentry{$\xi^-$}

\addplot [purple, mark size= 1pt, only marks, mark = *,mark options=solid] coordinates{(0,0)};
\addplot [
    domain=1.2:2.1, 
    samples=100, 
    color=purple,
    ]
    {2.5*x - 1.5};
 \addplot [purple, mark size= 1pt, only marks, mark = *,mark options=solid] coordinates{(1.2,2.5*1.2 - 1.5)};

\addplot [blue, mark size= 1pt, only marks, mark = *,mark options=solid] coordinates{(2.1,5*2.1-20)};
    \addplot [
    domain=3.6:4.5, 
    samples=100, 
    color=blue,
    ]
    {5*x - 23};
\addplot [blue, mark size= 1pt, only marks, mark = *,mark options=solid] coordinates{(3.6,5*3.6 - 23)};
\addplot [
    domain=4.5:6, 
    samples=100, 
    color=blue,
    ]
    {5*x - 5*6};
\addplot [blue, mark size= 1pt, only marks, mark = *,mark options=solid] coordinates{(4.5,5*4.5 - 30)};
    \addplot [
    domain=6:7, 
    samples=100, 
    color=purple,
    ]
    {2.5*x - 2.5*6};
 \addplot [purple, mark size= 1pt, only marks, mark = *,mark options=solid] coordinates{(6,0)};
    \addplot [
    domain=7:8.5, 
    samples=100, 
    color=purple,
    ]
    {2.5*x - 2.5*6.5};
 \addplot [purple, mark size= 1pt, only marks, mark = *,mark options=solid] coordinates{(7,2.5*7 - 2.5*6.5)};
 \end{axis}
\end{tikzpicture}

\caption{Example of a trajectory of the process $\xi$.}
\label{plot:pathxi}
\end{figure}
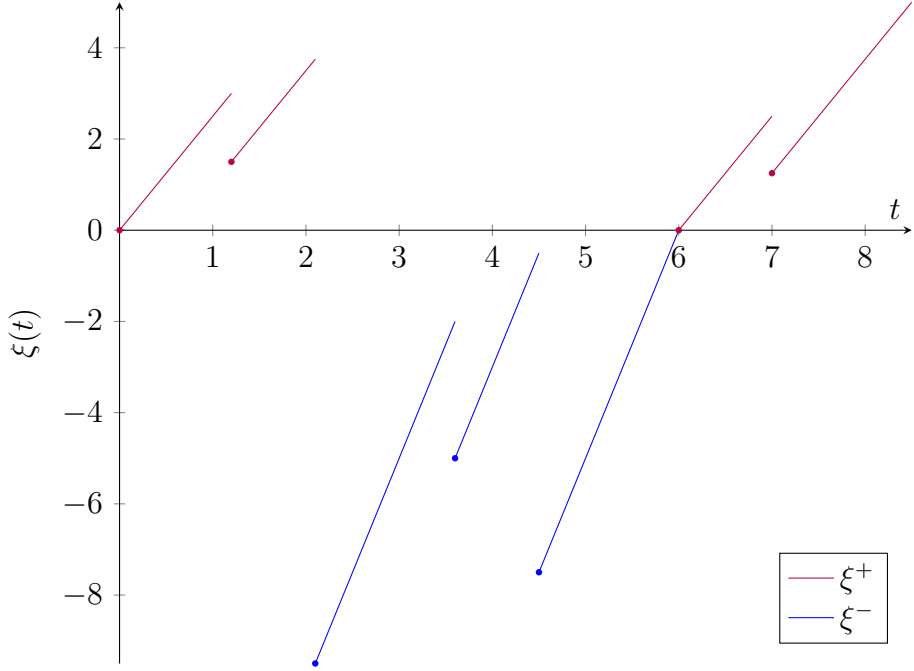
Similarly, we denote $\xi^+_t$ the Lévy process composed of a compound Poisson process having the same Lévy measure as $\xi^-_t$ plus a linear drift with rate $a_{_+}>0$. Its Laplace exponent is
\begin{equation}
\label{psip}
\Psi_{_+}(q)=  a_{_+}q + \int_0^1 (s^q -1) s \rho (s) ds.
\end{equation}
The evolution of $(\xi_t)_{t \geq 0}$ is the following: for any $t \geq 0$, if  $\xi_t \geq 0$, it evolves according to the law of $\xi^+_t$ and, if $\xi_t < 0$, it evolves according to the law of $\xi^-_t$. Since the process has bounded variation, the set $(- \infty, 0)$ is irregular for the process started at $0$, which means that $\xi$ spends a positive amount of time in $(0, \infty)$ before the first jump below $0$, when it changes the drift (see Figure \ref{plot:pathxi}). As a result, the excursions away from $0$ are simpler in this case than in the more general case of \eqref{sderefracted}. Excluding the degenerate case $\rho \equiv 0$, the process $\xi$ (and consequently $X$) is irreducible, which means that, for every $x >0$, the probability that the process starting from $x$ reaches a given point $y>0$ is strictly positive. 

Note that condition \eqref{conditiononrhonecessary} implies that
\begin{equation}
\int_{(-\infty, 0)} e^{-(1+\epsilon)x} \; \Pi(dx) = \int_{(-\infty, 0)} e^{(1-\epsilon)x} \rho(e^x) dx = \int_0^1 s^{-\epsilon} \rho(s) ds < + \infty.
\end{equation}
By Theorem $3.6$ in \cite{KYPR14}, this ensures that the Laplace exponents $\Psi_{_+}$ and $\Psi_{_-}$ are defined in the interval $[-1-\epsilon, \infty)$ and differentiable in $(-1-\epsilon, \infty)$. Their derivatives are given by
\begin{equation}
\label{psider}
\Psi'_{\pm}(q)= a_{_\pm}  +  \int_0^1 s^{q+1}  \log (s) \; \rho (s) ds.
\end{equation}

It is well known that if $\Psi'_{_-}(0) >0$, then $\xi^{-}$ drifts to $+\infty$, if $\Psi'_{_-}(0) =0$ then $\xi^{-}$ is recurrent and, if $\Psi'_{_-}(0) < 0$, then $\xi^{-}$ drifts to $- \infty$ and the same holds for $\xi^+$ in terms of $\Psi'_{_+}(0) $. 
A natural question that arises concerns the asymptotic behaviour of the refracted Lévy process $\xi$, if the behaviour of $\xi^+$ and $\xi^-$ is known. The following lemma answers to this question. Since the proof relies on the forthcoming Lemma \ref{lemmajointlaw}, we postpone it to Section \ref{Section3}.
\begin{lem}
\label{lemmaasymptoticbehaviourRLP}
Let $\xi^+$, $\xi^-$ and $\xi$ defined as above. Then the following hold:
\begin{enumerate}[label=(\roman*)]
\item \label{plusplus} If both $\xi^+$ and $\xi^-$ drift to $+\infty$, then $\xi$ drifts to $+\infty$. 
\item \label{minusminus} If both $\xi^+$ and $\xi^-$ drift to $-\infty$, then $\xi$ drifts to $-\infty$. 
\item \label{recplus} If $\xi^+$  is recurrent and $\xi^-$ drifts to $+\infty$, then $\xi$ is null recurrent.
\item  \label{minusrec}  If $\xi^+$  drifts to $-\infty$ and $\xi^-$ is recurrent, then $\xi$ is null recurrent.
\item \label{minusplus} If $\xi^+$  drifts to $-\infty$ and  $\xi^-$ drifts to $+\infty$, then $\xi$ is positive recurrent.
\end{enumerate}
\end{lem}
\begin{oss}
Even if it is formulated in terms of $\xi$, we believe that the previous lemma also applies to Lévy processes with unbounded variation. 
\end{oss}

The following lemma provides a Feynman-Kac representation of the semigroup $(T_t)_{t \geq 0}$ in terms of the Markov process $X$. The proof follows adapting the one of \cite[Lemma 2.2]{BW18} to our setting. Define
\begin{equation}
\mathcal{E}_t \coloneqq  \exp \left( \int_0^t \underline{c}(X_s) ds  \right)=  \exp \left(  \int_0^t \left( a_{_+}\mathbf{1}_{\{\xi_s \geq 0\} } + a_{_-}\mathbf{1}_{\{\xi_s < 0\} } \right) ds  \right), \quad \quad t \geq 0.
\end{equation}

\begin{lem}
The growth-fragmentation semigroup $(T_t)_{t \geq 0}$ has a Feynman-Kac representation, which is
\begin{equation}
T_tf(x) = x \E_x \Bigg( \mathcal{E}_t \frac{f(X_t)}{X_t} \Bigg), \quad \quad x >0.
\end{equation}
\end{lem}
We point out that such representation can be viewed as a "many-to-one formula" in the setting of branching particle systems. Informally, the mean behaviour of the whole system is described only in terms of the evolution of a single particle, often referred to as the "tagged fragment", that evolves according to the law of $(X_t)_{t \geq 0}$. 

In order to study the asymptotic behaviour of $T_t$ as $t \to \infty$, we need to understand the limiting behaviour of $\E_x \left[ \mathcal{E}_t f(X_t)/X_t \right]$ using ergodicity arguments. \\
For $y >0$, we introduce
\begin{equation}
H(y) \coloneqq \inf \{ t>0 \; : \; X_t=y \} = \inf \{ t>0 \; : \; \xi_t=\log y \}.
\end{equation}
Given $x,y > 0$, define
\begin{equation}
L_{x,y} (q) \coloneqq \E_x \left(e^{-qH(y)} \mc E_{H(y)}, \; H(y) < \infty  \right), \quad \quad q \in \mb R.
\end{equation}
The function $L_{x,y} : \mb R \to (0, \infty] $ is non-increasing and convex, with  $\lim_{q \to \infty} L_{x,y} (q) = 0$ and $\lim_{q \to - \infty} L_{x,y} (q) = \infty$ (see Section $3$ in \cite{BW18}).
\begin{defn}
\label{defmalthusexp}
Let $x_0 > 0$. We call 
\begin{equation}
\label{malthusexponent1}
\lambda \coloneqq \inf \{ q \in \mathbb{R} \; : \; L_{x_0, x_0}(q) < 1 \}.
\end{equation}
the \textit{Malthus exponent} of the growth-fragmentation operator $\mc A$.
\end{defn}
In \cite[Proposition 3.1]{BW18} it was proved that if there exists $q \in \mathbb{R}$ and $x_0 > 0$ with $L_{x_0, x_0}(q) < 1$, then $L_{x, x}(q) < 1$ for all $x >0$. 
Hence, we denote $H \coloneqq H(1)$ and 
\begin{equation}
\label{DefinitionofL}
L(q) \coloneqq L_{1,1} (q) = \E_1 \big(e^{-qH} \mc E_{H}, \; H < \infty  \big), \quad \quad q \in \mb R.
\end{equation}
The proof of Theorem \ref{Teorema1} begins with the explicit computation of the Laplace transform $L(q)$, which is stated in the forthcoming Lemma \ref{LemmaL}. This computation relies heavily on some properties of the excursions of spectrally negative Lévy processes and on the fact that the hitting time processes are subordinators (non decreasing Lévy processes). The main results used in the proof of Lemma \ref{LemmaL}  are summarized in Lemma \ref{lemmajointlaw}. A more extensive formulation can be found in \cite[Chapter 7]{BLP}. \\
Once $L(q)$ is computed, Lemma \ref{lemmamalthusexponent} yields the Malthus exponent $\lambda$. 
The existence of a solution $\lambda$ to $L(\lambda) =1$ will enable us to define in \eqref{martingalemprime}  a remarkable martingale multiplicative functional $\mc M'$ of $X$.
By probability tilting, $\mc M'$ yields another Markov process $(Y_t)_{t \geq 0}$ which enjoys a much simpler, but deeper, connection with the growth-fragmentation semigroup $(T_t)_{t \geq 0}$.  

More specifically, we introduce the new probability measure $\tilde{\p}_x$ such that, if $(\mc F_t)_{t \geq 0}$ is the natural filtration of $(X_t)_{t \geq 0}$, it holds that
\begin{equation}
\label{tiltedproba}
\tilde{\p}_x(A)= \E_x[\mathbf{1}_A \mc M'_t ], \quad \quad \forall A \in \mc F_t.
\end{equation}
Since $\p_x$ is a probability law on the space of càdlàg paths, the same holds for $\tilde{\p}_x$. Denoting $Y= (Y_t)_{t \geq 0}$ the process with distribution $\tilde{\p}_x$ \footnote{This means that its finite-dimensional distributions are given in the following way. Let $0 \leq t_1 < \dots < t_n \leq t$, and $F: \mathbb{R}^n \to \mathbb{R_+}$. Then, 
\begin{equation}
\tilde{\E}_x [F(Y_{t_1}, \dots , Y_{t_n} )] = \E_x[\mc M'_t F(X_{t_1}, \dots , X_{t_n} ) ], \quad \quad x > 0.
\end{equation}
},
we have 
\begin{equation}
\label{formulasemigroupandy}
e^{- \lambda t}  T_t f (x) = \tilde{\E}_x \big( f(Y_t) \big).
\end{equation} Condition \eqref{existencehyp} is the necessary and sufficient condition for $Y$ to be positive recurrent and then the classical ergodic theory for Markov processes readily leads to Theorem \ref{Teorema1}.\\
The proof of Theorem \ref{Teorema1} shows that the asymptotic profile $\nu$ is given by the stationary distribution of the process $Y$, which is characterised in Section \ref{Section4}. We see in Proposition \ref{lemmadensityoccupationmeasure} that such distribution is absolutely continuous with respect to the Lebesgue measure and we give an explicit formula for computing it. The fact that $Y$ is the exponential of a refracted Lévy process, enables us to provide an expression for the density in terms of the scale functions $\tilde{W}_{_+}$ and $\tilde{W}_{_-}$, which are functions that appear in several problems concerning spectrally negative Lévy processes and will be defined in Section \ref{Section4}. A further step consists in finding a more explicit expression for the density, since, in general, the scale functions are not known explicitly but only through their Laplace transform. This is done in Lemma \ref{lemmadensityetapositive} and Lemma \ref{lemmabehaviourdensitynear0}.

\section{Proof of Theorem 1.1}
\label{Section3}
We start by recalling some properties of spectrally negative Lévy processes with bounded variation, that will be essential in the proof.

We consider a spectrally negative Lévy process $Z$ with Laplace exponent
$$ \Psi_Z(q) = q \left( d - \int_0^{\infty} e^{- q z} \Pi_{Z} ( - \infty, -z ) dz \right),$$
where $d >0$ is the drift and $\Pi_Z$ is the Lévy measure. Let $\Phi_Z$ be the right inverse of $\Psi_Z$.
\begin{lem}
\label{lemmajointlaw}
Define $J= \inf\{ t \geq 0 \; : \; Z_t \in (- \infty, 0)\}$.  Then the following hold.
\begin{enumerate}[label=(\roman*)]
\item $\p (J = \infty) = 0$ if $\Psi'_Z(0+) \leq 0$ and $\p (J = \infty) = \Psi'_Z(0+)/ d $ otherwise. 
\item Let 
\begin{equation}
b =\begin{cases} d & \mbox{if } \Psi'_Z(0+) \leq 0
\\ d-  \Psi'_Z(0+)  & \mbox{otherwise.}
 \end{cases}
\end{equation}
The joint law of the triplet $(Z_{J-}, Z_J, J)$ is determined on $\{0 \leq x \leq -y \} \times [0, \infty)$ by 
\begin{align}
\p ( Z_{J-} \in dx, Z_{J} - Z_{J-} \in dy, J \in dt \; | \; J < \infty ) = b^{-1} \exp \{ - \Phi_Z (0) x \} dx \Pi_Z(dy) \\ \times  \p (J \in dt \; | \; Z_{J-} = x),
\end{align}
where the conditional law of $J$ that appears above is the one of a subordinator with characteristic exponent $\Phi^{\natural}(q) = \Phi_Z (q) - \Phi_Z (0)$ at time $x$.  This means that
\begin{equation}
\E \left[ \exp\{-qJ\} \; \big| \; Z_{J-} = x  \right] = \exp \{ -x \Phi^{\natural}(q) \},
\end{equation}
for every $q$ for which $\Phi^{\natural}(q)$ is defined.
\end{enumerate}
\end{lem}
\begin{proof}
\begin{enumerate}[label=(\roman*)]
\item It is proven in Theorem $17$ in \cite[Chapter 7]{BLP}.
\item  By Theorem $17$ in \cite[Chapter 7]{BLP}, the law of the pair $(Z_{J-}, Z_J)$ is given on $\{0 \leq x \leq -y \}$ by 
$$\p ( Z_{J-} \in dx, Z_{J} - Z_{J-} \in dy \; | \; J < \infty ) = b^{-1} \exp \{ - \Phi_Z (0) x \} dx \Pi_Z(dy).$$
Moreover, as a direct consequence of \cite[Chapter $7$, Exercise $3$]{BLP} and Theorem 17(iii) \cite[Chapter $7$]{BLP}, we have that, under the conditional law $\p( \cdot  \; | \; Z_{J-} = x)$,  $J$ is distributed like a subordinator with characteristic exponent $\Phi^{\natural}(q)$ at time $x$. 
\end{enumerate}
\end{proof}

The first step in the proof of Theorem \ref{Teorema1} consists in finding an explicit expression for the Laplace transform $L$. 

Let $\Psi_{_-}$ and $\Psi_{_+}$ be as in \eqref{psim} and \eqref{psip} and denote by $\Phi_-$ and $\Phi_+$ their right inverses. 
\begin{lem}
\label{LemmaL}
Define
\begin{equation}
\label{definitionqstar}
q^* \coloneqq \max{  \{ \inf_q   { \{ \Psi_-(q) ) \} } + a_{_-}, \inf_q { \{ \Psi_+(q) \} }  + a_{_+}  \} }.
\end{equation}
Then, for $q > q^*$, it holds that 
\begin{equation}
\label{formulaforL}
L(q)=  1 - \frac{(a_{_-} - a_{_+} ) (1+ \fim{q})}{a_{_+}  (\fip{q} - \fim{q})}.
\end{equation}
For $q < q^*$, $L(q)$ is infinite. Finally, $L(q^*)$ is finite if and only if both $\Psi_{_+}$ and $\Psi_{_-}$ reach their infimum.
\end{lem}

\begin{proof}
In the rest of the proof, we denote $\p_x$ and $\E_x$ to be the probability laws and the induced expectations of the process $\xi_t$ starting at $x \in \mathbb{R}$. We write $\E$ and $\p$ for $\E_0$ and $\p_0$. Let 
\begin{equation}
\label{defhittingtimenegativepart}
J= \inf \{ t > 0 \; : \; \xi_t \in (- \infty, 0] \}
\end{equation}
be the first hitting time of $(- \infty, 0]$. We note that, since the process is composed by a positive drift and a compound Poisson process, $J$ is strictly positive and the negative real line can be hit only after a jump. Thus, we denote $\xi_{J-} $ the position of the process before the jump in the negative part and $\xi_J $ the position after the jump. Obviously, $\xi_{J-} >0$ and $\xi_J < 0$. Finally, the size of the first jump to the negative part is given by $\Delta_{\xi_J}= \xi_J - \xi_{J-} $. 

\begin{figure}
\label{Fig2}
\centering
\begin{tikzpicture}
\begin{axis}[
    axis lines = left,
    xlabel = $t$,
    ylabel = {$\xi(t)$},
    axis x line=center,
    legend pos=north east
]

\addplot [purple, mark size= 2pt, only marks, mark = *,mark options={fill=white}] coordinates{(3.5,0)};
\node [above] at (axis cs:  3.5,  0) {$J$};
\addplot [blue, mark size= 2pt, only marks, mark = *,mark options=solid] coordinates{(6.5,0)};
\node [above] at (axis cs:  6.5,  0){$H$};
\addplot [purple, mark size= 2pt, only marks, mark = *,mark options=solid] coordinates{(3.5,5*3.5-28)};
\node [right] at (axis cs:  3.5,5*3.5-28){$\xi_J$};
\addplot [white, mark size= 1.5pt, only marks, mark = *,mark options=solid] coordinates{(3,5*3.5-29)};
\addplot [purple, mark size= 2pt, only marks, mark = *,mark options={fill=white}] coordinates{(3.5,2.5*3.5-1.5)};
\node [right] at (axis cs:  3.5,2.5*3.5-1.5){$\xi_{J^-}$};
\addplot [white, mark size= 1.5pt, only marks, mark = *,mark options=solid] coordinates{(3.5,2.5*3.5-1)};

\addplot [
    domain=0:1.5, 
    samples=100, 
    color=black,
]
{2.5*x};
\addplot [
    domain=3.5:5, 
    samples=100, 
    color=black,
    ]
    {5*x - 28};

\addplot [black, mark size= 1pt, only marks, mark = *,mark options=solid] coordinates{(0,0)};
\addplot [
    domain=1.5:3.5, 
    samples=100, 
    color=black,
    ]
    {2.5*x - 1.5};
    
 \addplot [black, mark size= 1pt, only marks, mark = *,mark options=solid] coordinates{(1.5,2.5*1.5-1.5)};
    \addplot [
    domain=5:6.5, 
    samples=100, 
    color=black,
    ]
    {5*x - 5*6.5};
    \addplot [black, mark size= 1pt, only marks, mark = *,mark options=solid] coordinates{(5, 5*5- 5*6.5)};
    
     \addplot [
    domain=6.5:7.5, 
    samples=100, 
    color=white,
    ]
    {2.5*x - 2.5*6.5};
 \end{axis}
\end{tikzpicture}

\caption{Decomposition of an excursion of $\xi$ at its first jump below $0$.}
\label{plot:firstjump}
\end{figure}
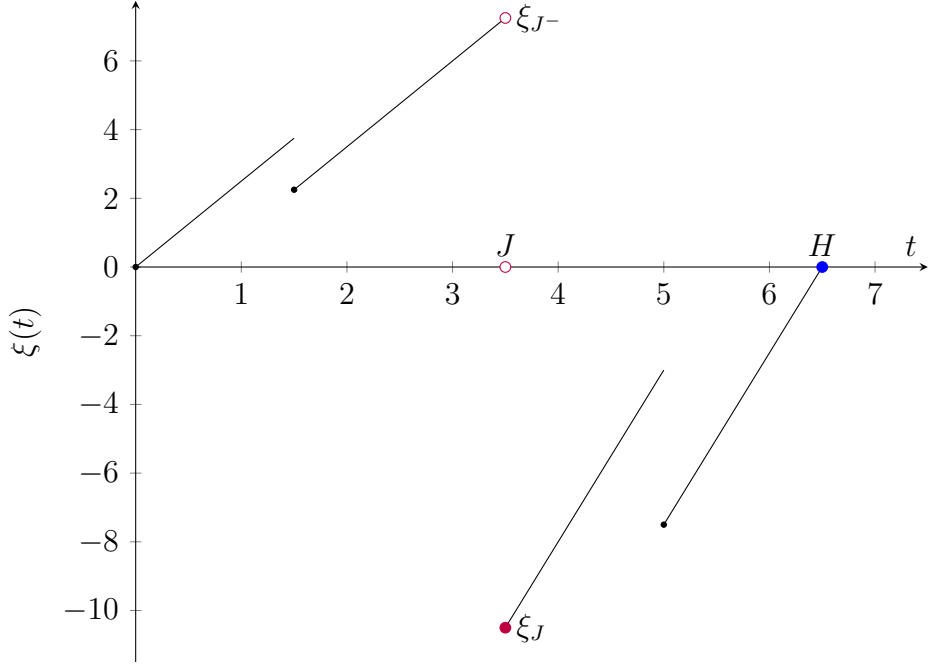

Clearly,  $J$ is a stopping time, and so

$$ \mathcal{E}_{H} = \exp \left( \int_0^{J} \underline{c}(\xi_s) ds +  \int_J^{H} \underline{c}(\xi_s) ds  \right) = \exp ( a_{_+} J ) \exp ( a_{_-} (H-J) )$$
and, for $q \in \mathbb{R}$,
\begin{equation}
\begin{split}
L(q) & = \mathbb{E}[ e^{-q H} \mathcal{E}_{H}, H < \infty ]   \\ 
& = \mathbb{E} [ e^{- (q- a_{_+}) J}  e^{- (q-a_{_-}) (H-J) } , J < \infty, H- J < \infty ] .
\end{split}
\end{equation}

Thanks to Lemma \ref{lemmajointlaw}, we can get an explicit expression after conditioning on $\xi_{J-}$ and $\xi_{J}$ and applying the Markov property to study the process $(\xi_{J+ t}$, $t>0)$ (see Figure \ref{plot:firstjump}). More precisely, $\p [\xi_{J-} \in dx, \Delta_ {\xi_J} \in dy  \; | \; J < \infty]  =  \Pi(dy) \; dx \; \mathbf{1}_{ \{0 \leq x \leq -y \}} b^{-1} e^{- \Phi_+(0)x} $, where 
\begin{equation}
b =\begin{cases} a_{_+} & \mbox{if } \Psi'_{_+}(0+) \leq 0
\\ a_{_+} -  \Psi'_{_+}(0+)  & \mbox{otherwise.}
 \end{cases}
\end{equation} 

Then,
\begin{equation}
\begin{split}
L(q) & = \E \left[ \mathbf{1}_{ \{J < \infty \}} \mathbb{E} \left[ e^{- (q- a_{_+}) J}  e^{- (q-a_{_-}) (H-J) } , H- J < \infty \; | \; J < \infty \right] \right]  \\ 
&  = \E \Bigg[ \mathbf{1}_{ \{J < \infty \}}  \int_0^{+ \infty} dx \; b^{-1}e^{- \Phi_+(0) x} \int_{- \infty}^{-x}  \Pi(dy) \;  \mathbb{E}[ e^{- (q- a_{_+}) J}  \; | \; \xi_{J-}= x, \Delta_ {\xi} = y, J < \infty]  \\
& \; \; \quad \quad \quad \quad \quad \times  \mathbb{E}[e^{- (q-a_{_-}) (H-J) } , H- J < \infty \; | \; \xi_{J-}= x, \Delta_ {\xi} = y]  \Bigg].
\end{split}
\end{equation}
Moreover, Lemma \ref{lemmajointlaw} and Theorem $1$ in \cite[Chapter $7$]{BLP} show that, under the conditional laws above, $J$ and $H-J$ are subordinators at respective time $x$ and $-x-y$, with characteristic exponents given by $\Phi_{_+}^{\natural}(q) = \Phi_{_+} (q) - \Phi_{_+} (0)$ and $ \Phi_{_-} (q)$, respectively. Recalling the definition of $q^*$ in \eqref{definitionqstar}, we see that, if $q < q^*$, at least one of the two characteristic exponents is not defined, meaning that one of the two inner expectations above is infinite. This easily implies that, if $q < q^*$, $L(q)=\infty$. \\
The same argument shows that $L(q^*)$ is finite if and only if $\Psi_+$ and $\Psi_-$ reach their infimum.

If $q > q^*$,
\begin{equation}
\begin{split}
L(q) & = \E \left[ \mathbf{1}_{ \{J < \infty \}}  \int_0^{+ \infty} dx \; b^{-1}e^{- \Phi_+(0) x} \int_{- \infty}^{-x} \Pi(dy)  e^{- (\Phi_+(q- a_{_+}) - \Phi_+ (0)) x }\cdot e^{\Phi_-(q-a_{_-})(x+y)}  \right]  \\
& = \frac{\p (J < \infty ) }{b}  \int_{- \infty}^{0} \Pi(dy) \; e^{\Phi_-(q-a_{_-})y} \;  \int_0^{-y}  dx \; e^{- x \left( \Phi_+(q- a_{_+})  - \Phi_-(q-a_{_-}) \right) } .
\end{split}
\end{equation}
By Lemma \ref{lemmajointlaw},
 $\p [ \{J < \infty \} ] / b = 1/a_{_+}$.
Thus, 
\begin{equation}
L(q) = \frac{\left[ \int_{- \infty}^{0} \;  \Pi(dy) \; e^{\fim{q} y} - \int_{- \infty}^{0} \;  \Pi(dy) \; e^{\fip{q} y}  \right]}{a_{_+}  (\fip{q} - \fim{q})}.
\end{equation}
Since $\Pi(dy)$ is finite, we can add and subtract  $\Pi((-\infty, 0))$ in order to express the numerator as 
\begin{equation}
\Psi_{_-}( \fim{q}) - a_{_-} \fim{q} -  \Psi_{_+}( \fip{q}) - a_{_+} \fip{q}
\end{equation}
and, by straightforward computations, we get the stated expression for $q > q^*$.
\end{proof}

\begin{oss}
In the case of a linear growth, that is $a_{_+}= a_{_-} = a$, we indeed obtain 
\begin{align}
L(q)& = \frac{\p (J <  \infty ) }{b}  \int_{- \infty}^{0} \Pi(dy) e^{ \Phi(q-a)y} \;  \int_0^{-y}  dx \; e^{- x \big( \Phi(q- a)  - \Phi(q-a) \big) }  \\
& = \frac{1}{a}  \int_{- \infty}^{0} (-y)  e^{\Phi(q-a)y} \;  \Pi(dy) = \frac{a - \Psi'(\Phi(q-a))}{a} \\ &= 1 - \frac{\Psi'(\Phi(q-a))}{a},
\end{align}
which coincides with the expression for $L(q)$ obtained by Bertoin and Watson in \cite[Section 6]{BW18}.
\end{oss}

It is convenient, at this point, to give the proof of Lemma \ref{lemmaasymptoticbehaviourRLP}.
\begin{proof}[Proof of Lemma \ref{lemmaasymptoticbehaviourRLP}]
First of all, \ref{plusplus} follows from the fact that $\xi$ solves $d \xi= d\xi^+ +  (a_{_-} - a_{_+}) \mathbf{1}_{ \{ Z_t < 0 \}} dt$, $\xi^+$ drifts to $+\infty$ and $a_{_-} - a_{_+} >0$.  Part \ref{minusminus} is shown in a similar way. \\
Let $\E$ and $\p$, $H$, $J$, $\xi_{J-} $, $\xi_J $ and $\Delta_{\xi_J}$ as as in the proof of Lemma \ref{LemmaL} (see Figure \ref{Fig2}). Clearly,  $H= J + (H-J)$. \\
Notice that in \ref{recplus}, \ref{minusrec} and \ref{minusplus}, $\xi$ is recurrent. Indeed, by Lemma \ref{lemmajointlaw}, $\p(J=\infty) = 0$ since $\xi^+$ is recurrent or drifts to $-\infty$. In addition, $\p(H-J=\infty) = 0$ when $\xi^-$ is recurrent or drifts to $+\infty$. Thus, $\p(H= \infty)=0$. \\ 
Now we prove \ref{recplus}. Conditioning on the law of the position before and after the jump at time $J$, we have by Lemma \ref{lemmajointlaw} that
\begin{align}
\E[J] & = \frac{1}{a_{_+}} \int_0^{+ \infty} dx \; e^{- \Phi_+(0) x} \int_{- \infty}^{-x}  \Pi(dy) \;  \mathbb{E}[ J  \; | \; \xi_{J-}= x]  \\
& = \frac{1}{a_{_+}} \int_0^{+ \infty} dx \; e^{- \Phi_+(0) x} \int_{- \infty}^{-x}  \Pi(dy) \; \left( - \frac{d}{dq} \E \left[ e^{-q J} ; | \; \xi_{J-}= x \right] \right)\Bigg|_{q=0} \\
& = \frac{1}{a_{_+}} \int_0^{+ \infty} dx \;  e^{- \Phi_+(0) x} \int_{- \infty}^{-x}  \Pi(dy) \; \left(- x \Phi'_{_+}(0)\right),
\end{align}
where the last identity comes from the fact that under the conditional law $\p(\cdot \; | \; \xi_{J-}= x, \Delta_ {\xi} = y)$, $J$ is a subordinator with characteristic exponent $\Phi_+$. We have that $\Phi'_{_+}(0)= 1/\Psi'_{_+}(0)=  \infty$, since $\xi^+$ is recurrent ($\Psi'_{_+}(0) = 0$). We conclude  that $\E[J]= \infty$ and so $\E[H]= \infty$. \\
The proof of \ref{minusrec} is similar. From above, we have that 
\begin{align}
\E[J] = \frac{1}{a_{_+}} \int_0^{+ \infty} dx \;  e^{- \Phi_+(0) x} \int_{- \infty}^{-x}  \Pi(dy) \; \left(- x \Phi'_{_+}(0)\right).
\end{align}
In this case, $\Phi'_{_+}(0)$ is negative and finite, leading to $\E[J]< \infty$. Similarly, conditioning on the jump at time $J$, we have that
\begin{align}
\label{expectedhminusj}
\E[H-J] & = - \frac{1}{a_{_+}} \int_0^{+ \infty} dx \;  e^{- \Phi_+(0) x} \int_{- \infty}^{-x}  \Pi(dy) \; \frac{(x+y)}{\Psi'_{_-}(0)},
\end{align}
where the last equality comes from the fact established in the previous proof that under the conditional law $\p(\cdot \; | \; \xi_{J-}= x, \Delta_ {\xi} = y)$, $H-J$ is a subordinator with characteristic exponent $\Phi_{_-}$ at time $-(x+y)$. Since $\xi^-$ is recurrent, the right-hand side is infinite and we conclude that $\E[H-J]=\infty$ and so $\E[H]=\infty$. Thus, $\xi$ is null recurrent. \\
Finally, \ref{minusplus} follows from the proofs of \ref{recplus} and \ref{minusrec}. As in \ref{minusrec}, $\E[J] < \infty$ since $\xi^+$ drifts to $-\infty$. Moreover, $\E[H-J] < \infty$ since $\xi^-$ drift to $+ \infty$ and $\Psi'_{_-}(0)$ ensures that $\int_{(-\infty, 0)} y \Pi(dy) < \infty$. Hence, $\E[H] < \infty$ and $\xi$ is positive recurrent.
\end{proof}

Having an explicit expression for $L$, we are able to determine the Malthus exponent $\lambda = \inf \{ q \in \mathbb{R} \; : \; L(q) < 1 \}.$
\begin{lem}
\label{lemmamalthusexponent}
Let \begin{equation}
\lambda^* = \int_0^1 \left( 1-s \right) \rho (s) ds.
\end{equation}
\begin{enumerate}[label=(\roman*)]
\item  \label{item1lemmamalthusexp}The Malthus exponent satisfies $\lambda \leq \lambda^*$.
\item \label{item2lemmamalthusexp} $L(\lambda)= 1$ if and only if 
\begin{equation}
\label{existencehypweak}
a_{_+} \leq - \int_0^1 \log (s) \;  \rho(s) ds \leq a_{_-}
\end{equation} In this case, $\lambda = \lambda^*$. 
\item \label{item3lemmamalthusexp} If the inequalities in \ref{item2lemmamalthusexp} are strict, i.e. under condition \eqref{existencehyp}, then $|L'(\lambda)|< + \infty$ and there exists $q< \lambda$ such that $1 < L(q) < \infty$.
\end{enumerate}
 \end{lem}

\begin{proof}
We start by proving \ref{item2lemmamalthusexp}. Recall by Lemma \ref{LemmaL}, that for $q > q^*$,
 \begin{equation}
L(q)=  1 - \frac{(a_{_-} - a_{_+} ) (1+ \fim{q})}{a_{_+}  (\fip{q} - \fim{q})}.
\end{equation}
For $L(q)$ to be equal to $1$, one must have either that 
\begin{equation}
\label{phiequaltom1} \fim{q}= -1,
\end{equation}
or that the denominator explodes. However, this second case cannot occur neither when $q>q^*$ by definition of $q^*$ nor at $q^*$ when $L(q^*) < \infty$, by Lemma \ref{LemmaL}. \\
By definition of $\Phi_{_-}$, for a solution to \eqref{phiequaltom1} to exist, one must have that $\Psi_-'(-1) \geq 0$. First of all, $\Psi_-'(-1)$ is well-defined and finite thanks to \eqref{conditiononrhonecessary}. Moreover, by \eqref{psider}, $\Psi_-'(-1) \geq 0$ is equivalent to $- \int_0^1 \log (s) \;  \rho(s) ds \leq a_{_-}$ and in this case the solution to \eqref{phiequaltom1} is given by 
\begin{equation}
\lambda^* = a_{_-} + \Psi_- (-1) = \int_0^1 \left( 1-s \right) \rho (s) ds.
\end{equation}
We need to analyse the denominator at $\lambda^*$. Note that $\fip{\lambda^*} = \Phi_{_+} (\Psi_{_+} (-1))$. Hence, the denominator is non-zero if and only if  $ \Phi_{_+} (\Psi_{_+} (-1)) \neq -1$.  By convexity of $L$, this is equivalent to $\Phi_{_+} (\Psi_{_+} (-1))> -1$, \textit{i.e.} $\Psi_+'( -1) < 0$. From \eqref{psider}, we have $a_{_+} < -\int_0^1 \; ds \;  \rho(s) \log (s)$. 
If $ \Phi_{_+} (\Psi_{_+} (-1)) = -1$, the denominator is $0$. In this case, looking at the second order, we see that 
\begin{equation}
\left( 1 + \fim{\lambda^*} \right)' = \frac{1}{\dpsim{-1}} > 0
\end{equation}
and 
\begin{equation}
\left( \fip{q} - \fim{q} \right)' = \frac{\dpsim{-1} - \dpsip{-1}}{\dpsip{-1} \dpsim{-1}}.
\end{equation}
The denominator is then infinite when $\dpsip{-1} = 0$ and positive when $\dpsip{-1} >0$. This implies that $L(\lambda^*)=1$ also in the boundary case $a_{_+} = -\int_0^1 \; ds \;  \rho(s) \log (s)$. \\
\ref{item3lemmamalthusexp} One can check that
\begin{equation}
\label{derivativeLlambda}
L'(\lambda)=  - \frac{a_{_-} - a_{_+}}{a_{_+}} \frac{1}{\Psi_-'( -1)} \frac{1}{(\Phi_+ (\Psi_+ (-1)) + 1)}.
\end{equation}
Thus, $L'$ is finite and negative as long as \eqref{existencehyp} holds, \emph{i.e.}, if $\Psi_-'( -1) > 0$ and $\Psi_+'( -1) < 0$.  Combined with \eqref{conditiononrhonecessary}, this leads the existence of a solution $q$ to $\fim{q} < -1$, or, equivalently, to $1< L(q) < \infty$. 

\ref{item1lemmamalthusexp} By convexity, it suffices to show that $L(\lambda^*) \leq 1$. We already showed that if \eqref{existencehyp} holds, then $L(\lambda^*)=1$. So, we need to analyse the remaining cases. If $\Psi_{_+}'(-1) > 0$, the, computing the limit, 
\begin{equation}
\lim_{q \to \lambda^*} L(q) =1 - \frac{a_{_-} - a_{_+}}{a_{_+}} \frac{\Psi_{_+}'(-1)}{\Psi_{_-}'(-1) - \Psi_{_+}'(-1)} = 1- \frac{\Psi_{_+}'(-1)}{a_{_+}} < 1. 
\end{equation}
If $\Psi_{_-}'(-1) < 0$, using the identity $\Psi_{_-}(q) - \Psi_{_+}(q) = (a_{_-} - a_{_+})q$, we get that $L(\lambda^*) < 1$.
\end{proof}

\begin{oss}
It is important at this point to make a connection with the spectral analysis approach \eqref{eigenelements} to this problem. Since the fragmentation kernel does not depend on the size, the function $h(x)=1$ is an eigenvector with eigenvalue $\lambda^*$ for the growth-fragmentation operator $\mc A$ defined in \eqref{gfoperator}. However, even though $h$ is not negative, this is not sufficient to state that $\lambda^*$ is the principal eigenvalue, since this would require to solve the eigenvalue problem for the dual operator $\mc A^*$ and check the positivity of the dual eigenfunction.
\end{oss}

We define the martingale
\begin{equation}
\label{martingalemprime}
\mc M_t' \coloneqq \frac{X_0}{X_t} \mc E_t e^{- \lambda^* t}, \quad \quad t \geq 0.
\end{equation}
The proof that $\mc M'$ is a martingale is similar to the proof of Lemma $7.2$ of \cite{BW18} and follows from the fact that decomposing the trajectory of $X$ at its jump times we have the expression
\begin{equation}
\mc E_t = \frac{X_t}{X_0} \prod_{0 < s \leq t} \frac{X_{s-}}{X_s}.
\end{equation}
and from some properties of Poisson point processes.

We introduce the new probability measure $\tilde{\p}_x$ obtained by tilting the law of $(X_t)_{t\geq 0}$ with $\mc M'$ (see \eqref{tiltedproba}) and we denote 
$Y= (Y_t)_{t \geq 0}$ the process with distribution $\tilde{\p}_x$. 
As stated in the Introduction, the proof of Theorem \ref{Teorema1} is based on the fact that $Y$ enjoys a simpler connection with the semigroup $T$, namely
\begin{equation}
e^{- \lambda^* t}  T_t f (x) = \tilde{\E}_x \big( f(Y_t) \big).
\end{equation} 
In the following Lemma, it will become clear that condition \eqref{existencehyp} is exactly the necessary and sufficient condition for the process $(Y_t)_{t \geq 0}$ to be positive recurrent and Theorem \ref{Teorema1} will follow from the classical ergodic theory for Markov processes. Recall that under this condition, by Lemma \ref{lemmamalthusexponent}, the Malthus exponent defined in \eqref{malthusexponent1} coincides with $\lambda^*$.
\begin{lem}
\label{lemmaY} The following hold.
\begin{enumerate}[label=(\roman*)]
\item \label{GeneratorofY} The process $Y$ is Markovian and its generator is given by 
\begin{equation}
\mathcal{G}_Y g(x)  = 
\begin{cases}
a_{_-} x g'(x) + \int_0^1 (g(sx)-g(x))  \rho(s) ds   \quad 	\quad x <1  \\
a_{_+} x g'(x) + \int_0^1 (g(sx)-g(x))  \rho(s) ds  \quad   \quad   x \geq 1.
\end{cases}
\end{equation}
\item \label{definitioneta} As a consequence, there exists a refracted Lévy process $\eta$ such that 
\begin{equation}
(Y_t)_{t \geq 0}= \left( e^{\eta_t} \right)_{t \geq 0}.
\end{equation}
\item The process $Y$ is recurrent if and only if \eqref{existencehypweak} holds. More precisely, $Y$ is positive recurrent if and only if  \eqref{existencehyp} holds and null recurrent if either $a_{_+}$ or $a_{_-}$ is equal to $- \int_0^1 \log (s) \;  \rho(s) ds$.
\end{enumerate}
\end{lem}
\begin{proof}
\begin{enumerate}[label=(\roman*)]
\item It follows from the properties of Poisson point processes, we refer to \cite[Section 7]{BW18} for more details.  
\item The process $\eta$ can be constructed in the same way as the process $\xi$. More precisely, $\eta$ has the same drifts as $\xi$ but a different Lévy measure, given by
\begin{equation}
\tilde{\Pi}(dz) = e^{z} \rho (e^z) dz, \quad \quad z < 0.
\end{equation}
The two underlying Lévy processes $\eta^-_t$ and $\eta^+_t$ have Laplace exponents $\tilde{\Psi}_{\pm} (q)$, given by the Lévy-Khintchine formula 
\begin{equation}
\label{psietapm}
 \tilde{\Psi}_{_\pm}(q) = a_{_\pm}q + \int_0^1 (s^q -1)  \rho (s) ds. 
\end{equation}
Observe that
\begin{equation}
\tilde{\Psi}_{\pm} (q) = \Psi_{\pm} (q -1) - \Psi_{\pm} (-1).
\end{equation} 
Differentiating \eqref{psietapm}, we see that, if the first inequality of \eqref{existencehyp} holds, the process $\eta^{+}$ drifts to $- \infty$ since
\begin{equation}
(\tilde{\Psi}_{_+} )'(0)=  \Psi '_{_+} (-1) < 0, 
\end{equation}
while it is recurrent if $a_{_+} = - \int_0^1 \log (s) \rho(s) ds$.
Similarly, the process $\eta^{-}$ drifts to $+ \infty$ when the second inequality in \eqref{existencehyp} holds, since
\begin{equation}
(\tilde{\Psi}_{_-} )'(0)=  \Psi '_{_-} (-1) > 0
\end{equation}
while it is recurrent when $ \Psi '_- (-1) = 0$.
\item It follows directly from the proof of \ref{definitioneta} and Lemma \ref{lemmaasymptoticbehaviourRLP}. In fact, if either $a_{_+} $ or $a_{_-}$ is equal to $- \int_0^1 \log (s) \;  \rho(s) ds$, then either $\eta^+$ or $\eta^-$ is recurrent and Lemma \ref{lemmaasymptoticbehaviourRLP} shows that $\eta$ is null recurrent. If both the inequalities are strict, then $\eta^+$ drifts to $-\infty$ and $\eta^-$ drifts to $+\infty$ and so $\eta$ is positive recurrent.
\end{enumerate}
\end{proof}

We are ready to prove Theorem \ref{Teorema1}. 
\begin{proof}
Recall that, under \eqref{existencehyp}, $\lambda = \lambda^*$.
\begin{enumerate}[label=(\roman*)]
\item Let $x>0$. By definitions \eqref{tiltedproba} and \eqref{martingalemprime}, we have
\begin{equation}
\tilde{\E}_x \left( f(Y_t) \right) = \E_x \left( \frac{X_0}{X_t} \mc E_t e^{- \lambda t} f(X_t) \right) = e^{- \lambda t} x \E_x \left( \mc E_t \frac{f(X_t)}{X_t} \right).
\end{equation}
Hence, 
\begin{equation}
e^{- \lambda t}  T_t f (x) = \tilde{\E}_x \big( f(Y_t) \big).
\end{equation}
By Lemma \ref{lemmaY}, under condition \eqref{existencehyp}, the process $Y$ is positive recurrent, hence it has a unique stationary distribution $\nu$ and it holds
\begin{equation}
\label{convergencetostationarydistributionY}
\lim_{t \to \infty} \tilde{\E}_x \big( f(Y_t) \big) = \langle \nu, f \rangle.
\end{equation}
To show that the convergence holds exponentially fast, note that, for $\delta >0$,
\begin{equation}
\tilde{\E}_x \left[ e^{\delta H_Y(x)} \Big] = \E_x \Big[ \mc E_{H_Y(x)} e^{- (\lambda - \delta) H_Y(x)}, H_Y(x) < + \infty \right] = L_{x,x}(\lambda - \delta).
\end{equation}
Lemma \ref{lemmamalthusexponent}{\ref{item3lemmamalthusexp}} shows that, for $\delta$ small enough, $L(\lambda - \delta) < \infty$ (and thus $L_{x,x} (\lambda - \delta) < \infty$ for all $x>0$), which proves that $Y$ is exponentially recurrent.
Using Kendall's renewal theorem (see Chapter $15$ in \cite{MT09}), it is well-known that this ensures that the above convergence \eqref{convergencetostationarydistributionY} is exponentially fast.
\item The second part of the Theorem stems from \cite[Theorem $1.1$]{BERT18}, as the author proved that the sufficient condition for the Malthusian behaviour is also necessary. 
\end{enumerate}
\end{proof}

\section{Asymptotic Profile}
\label{Section4}
\label{sectionproof2}

From the proof of Theorem \ref{Teorema1}, it follows that the asymptotic profile $\nu$ is given by the unique invariant distribution of $Y$, which is characterised in the following proposition.  

\begin{prop}
\label{lemmadensityoccupationmeasure} Assume that \eqref{conditiononrhonecessary}, \eqref{growthrate} and \eqref{existencehypweak}
hold.
\begin{enumerate}[label=(\roman*)]
\item \label{prop1} The unique (up on a constant factor) invariant measure $m(dx)$ of  $Y$ is absolutely continuous with respect to the Lebesgue measure and its density is locally integrable, everywhere positive and given by
\begin{equation}
\label{densityoccm}
m(x) = \frac{q(1, x)}{c(x) \; q(x, 1)} \quad	\quad x>0, 
\end{equation}
where $q(z, x) \coloneqq \tilde{\p}_z \left(H_Y(x) < H_Y(z) \right)$.
\item \label{prop2} If, moreover, \eqref{existencehyp} holds, $(Y_t)_{t \geq 0}$ is positive recurrent and its stationary distribution is 
\begin{equation}
\nu(dx) = \frac{m(dx)}{\langle m, \mathbf{1} \rangle},
\end{equation}
where the total mass of $m(dx)$ is given by
\begin{equation}
\label{totalmass}
\langle m, \mathbf{1} \rangle = \frac{a_{_-} - a_{_+}}{a_{_+}} \frac{1}{\Psi_-'( -1)} \frac{1}{(\Phi_+ (\Psi_+ (-1)) + 1)} < \infty.
\end{equation}
\end{enumerate}
\end{prop}

\begin{proof}
The proof of \ref{prop1} follows adapting Lemma $5.2$ in \cite{BW18} to our setting.  
We proved in Theorem \ref{Teorema1} that $(Y_t)_{t \geq 0}$ is recurrent if and only if \eqref{existencehypweak} holds and positive (and also exponentially) recurrent if and only if condition \eqref{existencehyp} holds. For the proof of \eqref{totalmass}, we recall (see \cite[Lemma 5.2]{BW18}) that the unique invariant measure of the process $Y$ is its occupation measure, defined as
\begin{equation}
\label{occupationmeasure}
\langle m, f \rangle \coloneqq \tilde{\E}_{x_0} \left( \int_0^{H_Y(x_0)} f(Y_s) ds \right),
\end{equation}
where $H_Y(x) = \inf \{ t > 0 \; : \; Y_t = x \}$. As a result, its mass is given by
\begin{equation}
\langle m, \mathbf{1} \rangle \coloneqq \tilde{\E}_{1} \left[ H_Y(1) \right].
\end{equation}
Recalling that 
\begin{equation}
\frac{\text{d} \tilde{\p}}{\text{d} \p} \Big|_{\mathcal{F}_t} = \mathcal{M}'_t,
\end{equation}
with a change of probability measure, we have
\begin{equation}
\langle m, \mathbf{1} \rangle =  \E_1 \left(H \; \mc M'_{H}, \; H < \infty  \right),
\end{equation}
where $H$ is the hitting time of the level $1$ for the process $X$.
Note that 
\begin{equation}
\mc M_H' \coloneqq \frac{X_0}{X_H} \mc E_H e^{- \lambda H},
\end{equation}
which implies
\begin{equation}
\begin{split}
\langle m, \mathbf{1} \rangle =  \E_1 \left(H \; e^{- \lambda H} \; \mc E_H, \; H < \infty  \right) = - L'(q).
\end{split}
\end{equation}
From \eqref{derivativeLlambda},
\begin{equation}
- L'(q) = \frac{a_{_-} - a_{_+}}{a_{_+}} \frac{1}{\Psi_-'( -1)} \frac{1}{(\Phi_+ (\Psi_+ (-1)) + 1)}.
\end{equation}
\end{proof}

Since $(Y_t)_{t \geq 0}= (e^{\eta_t})_{t \geq 0}$, it suffices to compute the invariant measure of $\eta$.  After a change of variables in \eqref{densityoccm}, its density, which we denote by $\overline{m}(y)$, can be expressed as
\begin{equation}
\label{densitymeasureeta}
\overline{m}(y) = \frac{q(1, e^{y})}{ \underline{c}(e^{y}) \; q(e^{y}, 1)} \quad	\quad y \in \mathbb{R},
\end{equation}
where
\begin{equation}
 \underline{c}(e^{y})  = \begin{cases}
 a_{_+} \quad \quad y \geq 0,\\
 a_{_-} \quad \quad y < 0.
 \end{cases}
\end{equation}
We use $\tp$ for the law of the process $\eta$ and $\te$ for the associated expectation and we denote 
\begin{equation}
H_{\eta}(y) \coloneqq \inf \{ t>0 \; : \; \eta_t=y \} .
\end{equation}
Thus, $q(1, e^{y}) = \tp_{0} \left( H_{\eta}(y) < H_{\eta} (0) \right)$ and a similar expression holds for $q(e^{y},1)$. To compute such quantities, we rely on the theory of scale functions. They appear in most of the results about boundary crossing problems and related path decompositions of spectrally negative Lévy processes. We refer to \cite{KKR12} for a comprehensive introduction about the topic.
\begin{defn}
For a given spectrally negative Lévy process $Z$, with Laplace exponent $\Psi_Z$, we define the scale function $W:  \mathbb{R} \to [0, \infty)$ as follows. We have $W(x)=0$ when $x < 0$, and otherwise on $[0, \infty)$, $W$ is the unique right continuous function whose Laplace transform is given by
\begin{equation}
\label{definitionscalefunctiongeneral}
\int_0^{\infty} e^{- q x} W(x) dx = \frac{1}{\Psi_Z(q)},
\end{equation}
for $q > \Phi_Z(0)$, where $\Phi_Z$ denotes the right-inverse of $\Psi_Z$. 
\end{defn}
One immediate application, which is central in our computations, is given by the so-called $\textit{two-sided exit problem}$, recalled in the following lemma. 
\begin{lem}
\label{lemmatwosidedexitproblem}
Let $Z$ a spectrally negative Lévy process and denote $\p$ and $\E$ its probability law and the corresponding expectation. For $a > 0$, define $$\tau_a^+ = \inf \{ t > 0 \; : \; Z_t > a \} \quad \text{and} \quad \tau_0^- = \inf \{ t > 0 \; : \; Z_t < 0 \}.$$
For all $x < a $, it holds that
\label{corollarytwosidedexit}
\begin{equation}
\p_x \left(\tau_a^+ < \tau_0^-\right) = \frac{W(x)}{W(a)}.
\end{equation} 
\end{lem}

We denote respectively $\tilde{W}_{_+}$ and $\tilde{W}_{_-}$ the scale functions associated to the Lévy processes $\eta_{_+}$ and $\eta_{_-}$. This means that $\tilde{W}_{_\pm} (x) = 0$ for $x < 0$ and 
\begin{equation}
\label{definitionscalefunction}
\int_0^{\infty} e^{- q x} \tilde{W}_{_\pm}(x) dx = \frac{1}{\tilde{\Psi}_{_\pm}(q) }.
\end{equation}
Let $\tilde{\Phi}_{_+}$ be the right inverse of $\tilde{\Psi}_{_+}$. We define
\begin{equation}
\label{bplus}
\bpl \coloneqq \tilde{\Phi}_{_+}(0).
\end{equation}
Since the process $\eta^+$ drifts to $-\infty$, $\bpl> 0$.
The first result expresses the density of the occupation measure of $\eta$ in a rather complicated form in terms of the scale functions $\tilde{W}_{_\pm}$.

\begin{lem}
\label{lemmadensitywithscalefunctions}
Let $y > 0$. The density of the invariant measure of the process $\eta$ is given by
\begin{equation}
\begin{split}
& \m(y)  = \frac{1}{a_{_+}} \left(a_{_+} \tilde{W}_{_+} (y) - \int_0^{+ \infty} e^{- \bpl x} dx  \int_{-x - y}^{-x} \tilde{\Pi}(dz) \tilde{W}_{_+} (y+x+z) \right)^{-1}, \\
& \m(-y)  = \tilde{W}_- (y) - \frac{1}{a_{_+}} \int_0^{+ \infty} e^{- \bpl x} dx  \int_{-x - y}^{-x} \tilde{\Pi}(dz) \tilde{W}_- (y+x+z).
\end{split}
\end{equation}
\end{lem}

\begin{proof}
From \eqref{densityoccm}, $\m (0)= 1/a_{_+}$, and $ \lim_{y \to 0^-} \m(y) =1/a_{_-}$.

Now consider $y > 0$. Then 
\begin{equation}
\m (-y) = \frac{1}{a_{_-}}\frac{\tp_0 (H_{\eta}(-y) < H_{\eta}(0))}{ \tp_{-y} (H_{\eta}(0) < H_{\eta}(-y))}.
\end{equation}

The event in the denominator depends on the path of $\eta$ while it evolves in the negative half-line, therefore it has the same law as $\eta^-$, that we denote by $\tilde{\p}^{^-}$. Shifting the Lévy process $\eta^-$ and applying Lemma \ref{lemmatwosidedexitproblem}, we get 
\begin{equation}
\tp_{-y} (H_{\eta}(0) < H_{\eta}(- y)) = \tilde{\p}_0^-(H_{\eta}(y) < H_{\eta}(0)) = \frac{\tilde{W}_- (0)}{\tilde{W}_- (y)}.
\end{equation}
On the other hand, using Lemma \ref{lemmajointlaw},
\begin{equation}
\begin{split}
\tp_0 (H_{\eta}(-y) < H_{\eta}(0)) & = \frac{1}{a_{_+}} \int_0^{+ \infty} e^{- \bpl x} dx \int_{-\infty}^{-x} \tilde{\Pi}(dz) \; \tilde{\p}^-_{x+z} (H_{\eta}(- y) < H_{\eta}(0))  \\
&= \frac{1}{a_{_+}} \int_0^{+ \infty} e^{- \bpl x} dx  \int_{-\infty}^{-x - y} \tilde{\Pi}(dz) \; \tilde{\p}^-_{x+z} (H_{\eta}(- y) < H_{\eta}(0)) \\
& \hspace{1.5cm} + \frac{1}{a_{_+}} \int_0^{+ \infty} e^{- \bpl x} dx  \int_{-x - y}^{-x} \tilde{\Pi}(dz) \;\tilde{\p}^-_{x+z} (H_{\eta}(- y) < H_{\eta}(0))  \\
& = 1 - \frac{1}{a_{_+}} \int_0^{+ \infty} e^{- \bpl x} dx  \int_{-x - y}^{-x} \tilde{\Pi}(dz) \frac{\tilde{W}_- (y+x+z)}{\tilde{W}_- (y)}.
\end{split}
\end{equation}
In the end, since $\tilde{W}_-(0) = 1/a_{_-}$, it follows
\begin{equation}
\m (- y)  = \tilde{W}_- (y) - \frac{1}{a_{_+}} \int_0^{+ \infty} e^{- \bpl x} dx  \int_{-x - y}^{-x} \tilde{\Pi}(dz) \tilde{W}_- (y+x+z).
\end{equation}

Now we analyse
\begin{equation}
\m(y) = \frac{1}{a_{_+}}\frac{\tilde{\p}_0 (H_{\eta}(y) < H_{\eta}(0))}{ \tilde{\p}_y (H_{\eta}(0) < H_{\eta}(y))}.
\end{equation}

Again by Lemma \ref{lemmatwosidedexitproblem}, the numerator is given by $\tilde{\p}_0 (H_{\eta}(y) < H_{\eta}(0)) = \tilde{W}_{_+} (0) / \tilde{W}_{_+}(y)$.
We now compute the denominator.
\begin{equation}
\begin{split}
\tilde{\p}_y (H_{\eta}(0) & < H_{\eta}(y))  = \tilde{\tilde{\p}}^-_0 (H_{\eta}(-y) < H_{\eta}(0) ) \\ 
& = \frac{1}{a_{_+}} \int_0^{+ \infty} e^{- \bpl x} dx \int_{-\infty}^{-x} \tilde{\Pi}(dz) \; \tilde{\p}^-_{x+z} (H_{\eta}(- y) < H_{\eta}(0)) \\
& = 1 - \frac{1}{a_{_+}} \int_0^{+ \infty} e^{- \bpl x} dx  \int_{-x - y}^{-x} \tilde{\Pi}(dz) \frac{\tilde{W}_{_+} (y+x+z)}{\tilde{W}_{_+} (y)}.
\end{split}
\end{equation}
Thus,
\begin{equation}
\m(y) 
= \frac{1}{a_{_+}} \left(a_{_+} \tilde{W}_{_+} (y) - \int_0^{+ \infty} e^{- \bpl x} dx  \int_{-x - y}^{-x} \tilde{\Pi}(dz) \tilde{W}_{_+} (y+x+z) \right)^{-1},
\end{equation}
which concludes the proof of Lemma \ref{lemmadensitywithscalefunctions}.
\end{proof}

The density $\overline{m}(y)$ for $y \geq 0$ can be expressed in a more explicit way. 
\begin{lem}
\label{lemmadensityetapositive}
The density $\overline{m}(y)$ of the invariant measure  of the process $(\eta_t)_{t \geq 0}$ is given on $\mathbb{R}_+$ by
\begin{equation}
\m(y) =  \frac{1}{a_{_+}} e^{- \bpl y}.
\end{equation}
\end{lem}

\begin{proof}
Let $y \geq 0$ and let $\mc L$ be the Laplace transform operator.

\begin{equation}
\begin{split}
\mc L \bigg( a_{_+ } & \frac{1}{\m} \bigg)(q)  = a_{_+} \int_0^{\infty} e^{-q y} \frac{1}{\m(y)} dy \\
& = \int_0^{\infty} dy \; e^{-q y} \left(a_{_+} \tilde{W}_{_+} (y) - \int_0^{+ \infty} e^{- \bpl x} dx  \int_{-x - y}^{-x} \tilde{\Pi}(dz) \tilde{W}_{_+} (y+x+z) \right).
\end{split}
\end{equation}
By \eqref{definitionscalefunction}, it follows that
\begin{equation}
a_{_+} \int_0^{\infty} dy \; e^{-q y}  \tilde{W}_{_+} (y) = \frac{a_{_+}}{\tilde{\Psi}_{_+}(q)}. 
\end{equation}
For the second term, 
\begin{equation}
\begin{split}
& \int_0^{\infty} dy \; e^{-q y} \left( \int_0^{+ \infty} dx \; e^{- \bpl x}   \int_{-x - y}^{-x} \tilde{\Pi}(dz) \tilde{W}_{_+} (y+x+z) \right) \\
& =  \int_0^{\infty} dy \; e^{-q y}\int_0^{+ \infty} dx \;  e^{- \bpl x}  \int_{- \infty}^0 \tilde{\Pi}(dz)\mathbf{1}_{ \{0 < y+x+z \} } \mathbf{1}_{ \{x < -z \} } \tilde{W}_{_+} (y+x+z) \\
& = \int_{- \infty}^0 \tilde{\Pi}(dz) \; e^{q z }  \int_0^{-z} dx \; e^{- (\bpl-q) x}   \int_0^{\infty} dh \; e^{-q h } \tilde{W}_{_+} (h)  \\
& =  \frac{1}{\tilde{\Psi}_{_+}(q)} \int_{- \infty}^0 \tilde{\Pi}(dz) \; e^{q z }  \int_0^{-z} e^{- (\bpl-q) x} dx. 
\end{split}
\end{equation}
Computing the integral and recalling the definition of the Laplace exponent $\tilde{\Psi}_{_+}$, we obtain that the second term is equal to
\begin{equation}
\begin{split}
  \frac{1}{\tilde{\Psi}_{_+}(q)(\bpl-q) }  \Bigg[ \int_{- \infty}^0 \tilde{\Pi}(dz) \big( e^{q z }& - 1  \big) - \int_{- \infty}^0 \tilde{\Pi}(dz)\big( e^{\bpl z} -1 \big)  \Bigg]  \\
& =  \frac{ \tilde{\Psi}_{_+}(q)  + a_{_+} (\bpl - q)}{\tilde{\Psi}_{_+}(q)(\bpl-q) }  = \frac{1}{\bpl- q} + \frac{a_{_+}}{\bpl}.
\end{split}
\end{equation}
As a result, 
\begin{equation}
\mc L \left(a_{_+}\frac{1}{\m} \right)(q) = \frac{1}{q-\bpl} = \mc L \left(e^{\bpl \cdot} \right)(q).
\end{equation}
Since the Laplace transforms is injective, the assertion follows.
\end{proof}

Changing the variables, we have the following result for the invariant distribution of the process $(Y_t)_{t \geq 0}$. Let 
\begin{equation}
\label{normalizingc}
c_1 \coloneqq \frac{1}{\langle m, \mathbf{1} \rangle}.
\end{equation}
\begin{prop}
\label{lemmaoccupationmeasurexmagg1}
The invariant distribution $\nu(dx)$ of the process $(Y_t)_{t \geq 0}$ on $(1, \infty)$ has a density given by 
\begin{equation}
\nu(x) = \frac{c_1}{a_{_+}} x^{-(1+\bpl)}.
\end{equation}
\end{prop}

Now we want to study the behaviour of $\tilde{m} (y) \coloneqq \m(-y)$, for $y> 0$. We recall that
\begin{equation}
\tilde{m} (y) = \tilde{W}_- (y) - \frac{1}{a_{_+}} \int_0^{+ \infty} e^{- \bpl x} dx  \int_{-x - y}^{-x} \tilde{\Pi}(dz) \tilde{W}_- (y+x+z).
\end{equation}

\begin{lem}
\label{lemmalaplacetransformnegativedensity}
 The Laplace transform of $\tilde{m}$ is 
\begin{equation}
\label{laplacetransformnegativedensity}
\mc L ( \tilde{m}  ) (q) = \frac{ \tilde{\Psi}_{_+}(q) }{a_{_+}\tilde{\Psi}_{_-}(q) (q- \bpl) }.
\end{equation}
\end{lem}

Unfortunately, in the negative part of the plane the Laplace transform cannot be inverted easily and it is not possible to obtain an expression for $\tilde{m}(y)$ as explicit as in Lemma \ref{lemmadensityetapositive}. However, we can describe its behaviour as $y \to \infty$. \\
In the following, we assume that there exists  $\bmin >0$ such that  
\begin{equation}
\label{Cramerscondition}
\tilde{\Psi}_{_-}(- \bmin) = 0.
\end{equation}
This condition is known in the literature as \textit{Cramér's condition}. 
\begin{lem}
\label{lemmabehaviourdensitynear0}
Suppose that \eqref{existencehyp} and \eqref{Cramerscondition} hold. 
Then, we have
\begin{equation}
\tilde{m}(y) \sim C e^{- y \bmin  } \quad \quad y \to \infty,
\end{equation}
with
\begin{equation}
\label{bigC}
C= - \frac{(a_{_-} - a_{_+}) \bmin}{a_{_+} \tilde{\Psi}'_{_-}(-\bmin) (\bmin + \bpl)}.
\end{equation}
\end{lem}
\begin{proof}  Recall that
\begin{equation}
 \tilde{\Psi}_{_\pm}(q) = a_{_\pm}q + \int_0^1 (s^q -1)  \rho (s) ds,
\end{equation}
and $\bpl > 0$.
The two functions $ \tilde{\Psi}_{_\pm}$ extend to holomorphic functions on a neighbourhood of $\{ - \bmin \leq \Re(z) < \infty \}$ in the complex plane. 
As a consequence, the right--hand side of \eqref{laplacetransformnegativedensity} defines a holomorphic function on the set $\{\Re(z) > - \bmin \}$ and therefore \eqref{laplacetransformnegativedensity} holds for any $q \in \mathbb{C}$ with $\Re(q) > - \bmin$. \\
Let $0 \leq \omega' \leq \bmin$. Then $ \tilde{\Psi}_{_-}$ vanishes only when $\omega' \in \{ 0, \bmin \}$ and $\tau = 0$ since
\begin{equation}
\begin{split}
\Re \left(\tilde{\Psi}_{_-}(- \omega' + i \tau) \right) & = a_{_-} (- \omega') + \int_0^1 (s^{- \omega'} \cos( \tau \ln s) -1)  \rho (s) ds \\
& < a_{_-} (- \omega') + \int_0^1 (s^{- \omega'} -1)  \rho (s) ds  \leq 0.
\end{split}
\end{equation}
For $0 < \omega' < \bmin$, $\frac{d}{dt} \left( e^{\omega' t} \tilde{m}(t) \right)$ is a tempered distribution with Fourier transform
\begin{equation}
\label{fouriertransformderivativeomegaprime}
\mathcal{F } \left( \frac{d}{dt} \left( e^{\omega' t} \tilde{m}(t) \right) \right) (\tau) = \frac{ i \tau \tilde{\Psi}_{_+}(- \omega' + i \tau) }{a_{_+}\tilde{\Psi}_{_-}(- \omega' + i \tau)) (- \omega' + i \tau  - \bpl) }.
\end{equation}
The right-hand side of \eqref{fouriertransformderivativeomegaprime} is a smooth and bounded function, thanks to the fact that $\tilde{\Psi}_{_-}(- \omega' + i \tau) =0$ if and only if $\omega' \in \{ 0, \bmin \}$ and $\tau = 0$. For the same reason, when $\omega' \uparrow \bmin$, it converges to the same expression with $\omega'$ replaced by $\bmin$. \\
We deduce that $f(t) \coloneqq \frac{d}{dt} \left( e^{\bmin t} \tilde{m}(t) \right)$ is a tempered distribution with everywhere smooth Fourier transform 
\begin{equation}
\label{fouriertransformderivativeomega}
\mathcal{F }f (\tau) = \frac{ i \tau \tilde{\Psi}_{_+}(- \bmin + i \tau) }{a_{_+}\tilde{\Psi}_{_-}(- \bmin + i \tau)(- \bmin + i \tau  - \bpl) }.
\end{equation}
We now set $k \coloneqq a_{_+}/a_{_-} \in (0, 1)$ and we notice that
\begin{equation}
\mathcal{F }f (\tau) = \frac{ i  k \tau  }{a_{_+} (- \bmin + i \tau  - \bpl) } + h(\tau),
\end{equation}
where
\begin{equation}
h(\tau) \coloneqq \frac{ i (1 - k) \tau \left( \int_0^1 (s^{- \bmin + i \tau } -1)  \rho (s) ds \right)}{a_{_+} \left( a_{_-}(- \bmin + i \tau) + \int_0^1 (s^{- \bmin + i \tau } -1)  \rho (s) ds \right) \left(- \bmin + i \tau  - \bpl \right) }.
\end{equation}
Since $h(\tau)= O(|\tau|^{-1})$ asymptotically at $\infty$, we have that $h \in L^2$ and, thus, $g(t)= \mathcal{F}^{-1}h \in L^2$. Also, we get 
\begin{equation}
\mc F (tg(t))(\tau) = i h'(\tau)= O(|\tau^{-1}|),
\end{equation}
and, thus, $tg(t) \in L^2$. Summing up, $(1+ |t|)|g(t)| \in L^2$ and thus $|g(t)| = (1+ |t|)^{-1})(1+ |t|) |g(t)| \in L^1$ by Cauchy-Schwarz. On the other hand, 
\begin{equation}
\frac{ i  k \tau  }{a_{_+} (- \bmin + i \tau  - \bpl) } = \frac{1}{a_{_-}} + \frac{\bmin +\bpl }{a_{_-} (- \bmin + i \tau  - \bpl)},
\end{equation}
whose inverse Fourier transform is $\frac{1}{a_{_-}} \delta_0 -  \frac{\bmin +\bpl }{a_{_-}}  e^{(\bmin +\bpl)} \mathbf{1}_{(- \infty, 0)}$. \\
So, in the distributional sense, 
\begin{equation}
\frac{d}{dt} \left( e^{\bmin t} \tilde{m}(t) \right)= \frac{1}{a_{_-}} \delta_0 -  \frac{\bmin +\bpl }{a_{_-}}  e^{(\bmin +\bpl)} \mathbf{1}_{(- \infty, 0)} + g(t).
\end{equation}
From this, we get that, on $(0, + \infty)$, $\tilde{m}$ is continuous and 
\begin{equation}
\begin{split}
\lim_{t \to + \infty} e^{\bmin t} \tilde{m}(t) & = \int_{\mathbb{R}} dt \left( \frac{1}{a_{_-}} \delta_0 -  \frac{\bmin +\bpl }{a_{_-}}  e^{(\bmin +\bpl)} \mathbf{1}_{(- \infty, 0)} + g(t) \right) \\
& = \int_{\mathbb{R}} g(t) dt = h(0) = \frac{(1 - \eta) \int_0^1 \rho (s) (s^{- \bmin} - 1) ds}{a_{_+} \tilde{\Psi}'_{_-}(-\bmin) (-\bmin - \bpl)} \\
& = - \frac{(1- \eta) a_{_-} \bmin}{a_{_+}\tilde{\Psi}'_{_-}(-\bmin)(\bmin + \bpl)} = - \frac{(a_{_-} - a_{_+}) \bmin}{a_{_+} \tilde{\Psi}'_{_-}(-\bmin) (\bmin + \bpl)}.
\end{split}
\end{equation}
\end{proof}

Recall that the total mass of $\langle m, \mathbf{1} \rangle$ and the constant $C$ have been computed in \eqref{totalmass} and \eqref{bigC}, respectively. Define
\begin{equation}
\label{normalizingc2}
c_2 \coloneqq \frac{C}{\langle m, \mathbf{1} \rangle}.
\end{equation}
\begin{prop}
\label{lemmaoccupationmeasurexmin1}
The density $\nu(dx)$ of the invariant distribution of $(Y_t)_{t \geq 0}$ is such that, when $x \to 0$, 
\begin{equation}
\nu(x) \sim c_2 x^{ - 1 + \bmin}.
\end{equation}
\end{prop}

\section{An example}
\label{section5}
We consider   
\begin{equation}
\label{rhomonomial}
\rho(s)=s^{\gamma -1}, 
\end{equation}
for some parameter $\gamma \geq 1$.
In particular, for $\gamma = 1$ we have the case in which the fragmentation is uniform.
Since 
\begin{equation}
\int_0^1 \log(s) \; s^{\gamma -1 } ds = - \frac{1}{\gamma^2},
\end{equation}
condition \eqref{existencehyp} summarizes as 
\begin{equation}
\label{boundsintheexample}
a_{_+} <  \frac{1}{\gamma^2}  < a_{_-}
\end{equation}
and, when it holds, the Malthus exponent is given by
\begin{equation}
\label{lambdaex}
\lambda = \int_0^1 (1-s) s^{\gamma -1 } ds = 
 \frac{1}{\gamma (\gamma + 1)}.
\end{equation}
The two Lévy processes $\eta^-$ and $\eta^+$ have Laplace exponents
\begin{equation}
\label{psiexpm}
\tilde{\Psi}_{_{\pm}} (q) = a_{_{\pm}} q + \int_0^1 (s^{q+ \gamma - 1} - s^{\gamma - 1}) ds = a_{_{\pm}} q - \frac{q}{\gamma (q+ \gamma)}.
\end{equation} 
We note that 
\begin{equation}
\tilde{\Psi}_{\pm} (q) = q \left (\frac{  a_{\pm} \gamma q+ a_{\pm}  \gamma^2 - 1}{\gamma (q+ \gamma)}  \right)
\end{equation} 
from which we deduce from \eqref{boundsintheexample} that 
\begin{equation}
\bpl = \frac{1 -  a_{_+} \gamma^2}{a_{_+} \gamma} > 0
\end{equation}
and 
\begin{equation}
- \bmin = \frac{1 -  a_{_-} \gamma^2}{a_{_-} \gamma} < 0.
\end{equation}
In this case the Cramér's condition is satisfied and, moreover, it is possible to invert the Laplace transform \eqref{laplacetransformnegativedensity} and obtain an explicit expression for the asymptotic profile also in the negative half-line. More precisely, we get
\begin{equation}
\begin{split}
\mc L ( \tilde{m}  ) (q) & = \frac{ a_{_+} \gamma q+ a_{_+} \gamma^2 - 1 }{a_{_+}  \left ( a_{_-} \gamma q+ a_{_-} \gamma^2 - 1\right) (q - \bpl)} \\ 
&= \frac{ a_{_+} \gamma ( q - \bpl)}{a_{_+}  (a_{_-} \gamma) ( q + \bmin) (q - \bpl)} \\
&= \frac{1}{a_{_-} (q + \bmin) },
\end{split}
\end{equation}
which is the Laplace transform of  $x \mapsto e^{-\bmin x}/a_{_-}$. We can conclude that, when $\rho  (s)= s^{\gamma - 1}$
\begin{equation}
\overline{m}(y)=\frac{1}{a_{_-}} e^{ \bmin y}, \qquad \quad y<0
\end{equation}
With a change of variables, for $x < 1$, the invariant measure of $Y$ is
\begin{equation}
m(dx) =\frac{1}{a_{_-}}  x^{-1 + \bmin}dx. 
\end{equation}
In this case, the total mass can be computed directly and one has
\begin{equation}
\langle m, \mathbf{1} \rangle = \frac{1}{a_{_-} \bmin} + \frac{1}{a_{_+} \bpl}.
\end{equation}
To sum up, defining $c_3 = 1/ \langle m, \mathbf{1} \rangle$, the asymptotic profile is
\begin{equation}
\nu(dx)= c_3 \left( \frac{1}{a_{_-}}  x^{-1 + \bmin} \mathbf{1}_{\{0 < x <1\}} +  \frac{1}{a_{_+}}  x^{-(1 + \bpl)} \mathbf{1}_{\{x \geq 1\}} \right) dx.
\end{equation}
In the case of uniform dislocations, \textit{i.e.} $\rho= 1$, we have
\begin{equation}
\nu(dx) = c_3 \left( \frac{1}{a_{_-}}  x^{- 1/a_{_-}}  \mathbf{1}_{\{0 < x <1\}} +  \frac{1}{a_{_+}}  x^{- 1/a_{_+}}  \mathbf{1}_{\{x \geq 1\}} \right) dx,
\end{equation}
with
$$c_3 = \left( \frac{1}{a_{_-} -1} + \frac{1}{1- a_{_+}}\right)^{-1}.$$

\bibliographystyle{plain}	
\bibliography{biblio}

\end{document}